\documentclass[a4paper]{article}
\usepackage[T1]{fontenc}
\usepackage[utf8]{inputenc}

\usepackage{graphicx}
\usepackage{tikz}
\usetikzlibrary{positioning}
\usepackage{lipsum}
\usepackage{epigraph}
\usepackage{amsmath}
\usepackage{braket}
\usepackage{amssymb}
\usepackage{amsfonts}
\usepackage{amsthm}
\usepackage{booktabs}
\usepackage{bbm}
\usepackage{mathtools}
\usepackage{hyperref}
\usepackage[capitalise]{cleveref}
\usepackage[thinc]{esdiff}

\theoremstyle{plain}
\newtheorem{thm}{Theorem}

\newtheorem{prop}[thm]{Proposition}

\newtheorem{lem}[thm]{Lemma}

\newtheorem*{claim*}{Claim}
\newtheorem*{question}{Question}
\theoremstyle{definition}
\newtheorem{defn}[thm]{Definition}
\newtheorem{rmk}[thm]{Remark}

\newcommand{\N}{\mathbb{N}}

\newcommand{\F}{\mathbb{F}}

\newcommand{\Z}{\mathbb{Z}}

\newcommand\aut{\operatorname{\mathsf{Aut}}}
\newcommand\rist{\operatorname{\mathsf{Rist}}}
\newcommand\stab{\operatorname{\mathsf{Stab}}}
\newcommand\tree{{\mathcal T}}

\title{Virtual first Betti number of GGS groups}
\author{Andrew Ng}

\begin{document}

\maketitle

\begin{abstract}
We observe a criterion for groups to have vanishing virtual first Betti number and use it to give infinitely many examples of torsion-free, finitely generated, residually finite groups which aren't virtually diffuse. This answers a question raised by Kionke and Raimbault.
\end{abstract}

\section{Introduction}
The \emph{Kaplansky unit conjecture} is the conjecture that for an integral domain $R$ and a torsion-free group $G$, any unit in the group ring $R[G]$ is of the form $rg$, for some $r \in R^{\times}$ and $g \in G$.
\begin{defn}
    A group $G$ has the \emph{unique product property} (or ``has unique products'', or ``has UP'') if for all non-empty finite subsets $A, B \subseteq G$ there exists $g \in G$ such that $g = ab$ for a unique pair $(a, b) \in A \times B$. 
    A group $G$ has the \emph{two unique products property} if for all finite subsets $A, B \subseteq G$ with $|A| |B| \geq 2$, there exist elements $g_0 \neq g_1$ of $G$ such that $g_0 = a_0 b_0$ for a unique pair $(a_0, b_0) \in A \times B$ and $g_1 = a_1 b_1$ for a unique pair $(a_1, b_1) \in A \times B$.

\end{defn}
In \cite{St80} it is shown that a group has the two unique products property if and only if it has the unique product property. It follows that a group which has UP satisfies the unit conjecture. In fact, all known methods for showing that a group satisfies the unit conjecture rely on showing that the group has unique products. The Kaplansky unit conjecture was disproved in \cite{Ga21}. 

\begin{defn}
    Let $A \subset G$ be a finite subset.
    An element $a \in A$ is called \emph{extremal} if for all $1 \neq s \in G$, either $as \notin A$ or $a s^{-1} \notin A$ (or both).
    A group is called \emph{diffuse} if every non-empty finite subset contains an extremal element.
\end{defn}

This property was introduced in \cite{Bow00}. Note that groups which are either diffuse or have the unique product property must be torsion-free, and in fact all diffuse groups have unique products \cite{Bow00}.

It is shown in \cite{Del97} that all residually finite hyperbolic groups are virtually diffuse.
Later work of \cite{KR16} showed that lattices in the rank 1 Lie groups $SO(n,1), SU(n,1),$ and $Sp(n,1)$ are virtually diffuse. They also provide a proof that all 3-manifold groups are virtually diffuse, which leads them to ask whether there exists a torsion-free, finitely generated, residually finite group which isn't virtually diffuse. We answer this in the affirmative:
\begin{thm}\label{notvirtdiff}
   For each odd prime $p$ there is a torsion-free, finitely generated, residually finite group whose finite quotients are all $p$ groups, which isn't virtually diffuse.
\end{thm}

\section{Vanishing virtual first Betti number}
In this section we observe a criterion for the virtual first Betti number of a group to vanish.
\begin{lem} \label{vb1}
    Let $q$ be a prime number and let $G$ be a finitely generated group whose finite quotients are all of order prime to $q$. Then $G$ has vanishing virtual first Betti number.
\end{lem}
\begin{proof}
    Suppose $G$ has a finite index subgroup $H$ with a surjective map $\phi: H \to \Z$. This gives a map $\phi_1: H \to \Z/q\Z$. The kernel of this map $N$ is a finite index subgroup of $G$, so contains a subgroup $N_1$ which is a finite index normal subgroup of $G$. $|G/N_1|$ is then divisible by $q$, which is a contradiction.
\end{proof}
\begin{defn}
A group $G$ is amenable if there exists a finitely additive translation in
variant probability measure on all subsets of $G$.     
\end{defn}
\begin{defn}
    A group is said to be \emph{locally indicable} if any non-trivial finitely generated subgroup surjects $\Z$.
\end{defn}

\begin{thm} \label{lidiff}   \cite[Theorem 6.4]{LW14}
    An amenable group is locally indicable if and only if it is diffuse.
\end{thm}
Note that the authors of \cite{LW14} call groups satisfying the above definition of diffuse 'weakly diffuse' and show that for all groups it is in fact equivalent to the stronger property of having two extremal elements.

\section{Groups acting on rooted trees}

In this section we give a brief introduction to groups acting on rooted trees, following the exposition of \cite{BG02}, and some commonly studied properties which will be used in the proof of \cref{notvirtdiff}.

Let $\Sigma$ be a finite set and consider the set $\Sigma^*$ of words in $\Sigma$. One can form a 
regular rooted tree $\tree$ where vertices correspond to elements of $\Sigma^*$ and there is an edge between vertices representing words $w_1$ and $w_2$ if and only if for some $\sigma \in \Sigma$, $w_1= w_2 \sigma$ or $w_2= w_1 \sigma$. Note the root is the empty word. 

Fix some $\Sigma$, and consider any subgroup $G<\aut(\tree_{\Sigma})$ of the group of automorphisms of the rooted tree. Let
$\stab_G(\sigma)$, the \emph{vertex stabiliser} of $\sigma$, denote the
subgroup of $G$ consisting of the automorphisms that fix the sequence $\sigma$, and let $\stab_G(n)$, the \emph{level stabiliser}, denote the subgroup of $G$ consisting of the automorphisms that fix all sequences
of length $n$:
\[\stab_G(\sigma)=\{g\in G|\,g\sigma=\sigma\},\qquad\stab_G(n)=\bigcap_{\sigma\in\Sigma^n}\stab_G(\sigma).\]
The $\stab_G(n)$ are normal subgroups of finite index of $G$; in
particular, letting $d:=|\Sigma|$, $\stab_G(1)$ is of index at most $d!$. Now note that the intersection of all level stabilisers is trivial, so $\aut(\tree_{\Sigma})$ and all its subgroups are residually finite. Let $G_n$ be the
quotient $G/\stab_G(n)$. If $g\in\aut(\tree_{\Sigma})$ is an automorphism fixing
the sequence $\sigma$, we denote by $g_{|\sigma}$ the element of
$\aut(\tree_{\Sigma})$ corresponding to the restriction to sequences starting
by $\sigma$:
\[\sigma g_{|\sigma}(\tau)=g(\sigma\tau).\]
As the subtree starting from any vertex is isomorphic to the initial
tree $\tree_\Sigma$, we obtain this way a map
\begin{equation}\label{eq:phi}
  \phi:\begin{cases}\stab_{\aut(\tree_{\Sigma})}(1)\to\aut(\tree_{\Sigma})^\Sigma\\
    h\mapsto (h_{|0},\dots,h_{|d-1})\end{cases}
\end{equation}
which is an embedding.  For a sequence $\sigma$ and an automorphism
$g\in\aut(\tree_{\Sigma})$, we denote by $g^\sigma$ the element of
$\aut(\tree_{\Sigma})$ acting as $g$ on the sequences starting by $\sigma$, and
trivially on the others:
\[g^\sigma(\sigma\tau)=\sigma g(\tau),\qquad g^\sigma(\tau)=\tau\text{ if }\tau\text{ doesn't start by }\sigma.\]
The \emph{rigid stabiliser} of $\sigma$ is
$\rist_G(\sigma)=\{g^\sigma|\,g\in G\}\cap G$. We also set
\[\rist_G(n)=\langle\rist_G(\sigma)|\,\sigma\in\Sigma^n\rangle=\prod_{|\sigma|=n}\rist_G(\sigma)\]
and call it the \emph{rigid level stabiliser} ($\prod$ denotes direct
product).  We say $G$ has \emph{infinite rigid stabilisers} if all
the $\rist_G(\sigma)$ are infinite.

\begin{defn} Let $G$ be a subgroup of $\aut(\tree_{\Sigma})$
\begin{itemize}
\item  $G$ is \emph{spherically transitive} if the
  action of $G$ on $\Sigma^n$ is transitive for all $n\in\N$.
    \item   $G$ is \emph{fractal} if for every vertex $\sigma$ of
  $\tree_\Sigma$ one has $\stab_G(\sigma)_{|\sigma}\cong G$, where the
  isomorphism is given by identification of $\tree_\Sigma$ with its
  subtree rooted at $\sigma$.
\item 
  $G$ has the \emph{congruence subgroup property} if every finite-index
  subgroup of $G$ contains $\stab_G(n)$ for some $n$ large enough.
\end{itemize}

\end{defn}
From now on assume $G$ acts spherically transitively.
\begin{defn}
      $G$ is a \emph{regular branch} group if it has a finite-index
    normal subgroup $N<\stab_G(1)$ such that
    \[N^\Sigma<\phi(N).\]
    It is then said to be \emph{regular branch over $N$}.
\end{defn}
\begin{defn}
     An automorphism $g \in \aut(\tree_{\Sigma})$ is \emph{finite state} if $S(g) := \{g|_v : v  \in \Sigma^*\}$ is finite. The element $g$ is \emph{bounded} if there is $N \geq 0$ such that for all $n \geq 1$, $|\{v \in \Sigma^n :  g|_v \neq 1\}| \leq N.$
\end{defn}
\begin{defn}
    A group $G \leq \aut(\tree_{\Sigma})$ is \emph{self-similar} if $g|_w \in G$ for
 all $w \in \Sigma^*$ and $g \in G$.
\end{defn}
\begin{defn}
    A group $G \leq \aut(\tree_{\Sigma})$ is said to be a \emph{bounded automata group} if $G$ is finitely generated and self-similar and every element $g \in G$ is bounded and finite state.
\end{defn}
By combining results of \cite{BKN10} and \cite{Ne05} all bounded automata groups are amenable.

\section{GGS groups}
In this section we introduce the Grigorchuk-Gupta-Sidki groups and prove \cref{notvirtdiff}. We will give two proofs, one which does not rely on CSP and one which does. 
Fix a prime $p$ and a non-zero vector $(e_1, e_2, \dots, e_{p-1}) \in \F_p^{p-1}$. 
Let $a:=(1,2,\dots, p)$ denote the element that rigidly permutes the subtrees according to the cyclic permutation. Equivalently, viewing the tree as words in the alphabet $\{x_1 \dots, x_p\}$, any non-empty word $w$ is of the form $x_iw_1$ for some $w_1$, and, taking indices mod $p$, $a(w) := x_{i+1}w_1$.
Define the tree automorphism $b$ recursively by $b:=<a^{e_1}, a^{e_2}, \dots a^{e_{p-1}},b>$. This means that on the $i^{\text{th}}$ subtree, the automorphism in the $i^{\text{th}}$ entry is applied. 
\\
The Grigorchuk-Gupta-Sidki group defined by $p$ and $(e_1, e_2, \dots, e_{p-1})$ is the group generated by the two elements $a$ and $b$. This class includes the Fabrykowski-Gupta group, introduced in \cite{FG91} as one of the first examples of a group of intermediate growth, and the Gupta-Sidki groups, introduced in \cite{GS83} as negative solutions to the Burnside problem.

\begin{proof} [First proof of \cref{notvirtdiff}]
In \cite[Theorem D]{FGU17} the authors give a criterion for a GGS group to be virtually torsion-free and show that for any $\lambda \neq 1,2$, the GGS group $G$ defined by the vector $(1, \dots , 1, \lambda)$ is just infinite and branch over its derived subgroup, which is torsion-free. 

By \cite[Theorem 5.2]{BGS05} the derived subgroup has finite abelianisation, hence isn't locally indicable. We show that the groups $G$ aren't virtually diffuse by directly appealing to the branch structure. Suppose the branching subgroup $K$ isn't diffuse, equivalently locally indicable. Since $d\geq 2$, each level contains an embedded copy of $K \times K$ Select a copy of $$K \times K$$ at the first level; the second copy of $K \times K$ then contains a subgroup isomorphic to $K \times K$. One may then iterate the construction using a copy of $K \times K$ from this subgroup. The copies of $K$ selected at each level then generate their infinite direct sum which embeds as a subgroup of $G$.
From the definition of GGS groups it follows that all such $G$ are bounded automata groups and hence are amenable, so cref{lidiff} then applies to show $K$ isn't diffuse and its infinite direct sum isn't virtually diffuse.  Since they are subgroups of the profinite group $\aut(\tree_{\Sigma})$, they are residually finite.
For completeness we include a proof that the infinite direct sum isn't virtually diffuse below. Since being virtually diffuse passes to subgroups, $G$ can't be virtually diffuse, so neither can any of its finite index torsion-free subgroups. This approach doesn't require the congruence subgroup property. 

\end{proof} 

\begin{prop}
    Let $G$ be a finitely generated non-diffuse group. Fix an infinite indexing set $I$ and for each $i \in I$ let $G_i$ be a group isomorphic to $G$. Then the infinite direct sum $S := \bigoplus_{I} G_i$ is not virtually diffuse.
\end{prop}
\begin{proof}
    Since any finite index subgroup $H$ contains a finite index normal subgroup $N$, it suffices to show that, for any finite group $Q$, the kernel $N$ of any homomorphism $\phi: S \to Q$ contains a copy of $G$. Since $G$ is finitely generated, there are only finitely many distinct homomorphisms $\psi_1, \dots \psi_k$ from $G$ to $Q$. Hence, there is some $\psi_j$ such that for infinitely many $i \in I$, $\phi|_{G_i} = \psi_j$.
    Choose distinct indices $i_1 \dots i_{|Q|}$ so that the induced homomorphism on $G_{i}$ is the same. Then the diagonal embedding $G \to \bigoplus_{k=1}^{|Q|} G_{i_k} \leq S$ has image in $N$. 
\end{proof}

It is worth noting, however, that the infinite free product $F := *_{I}  G_i$ \emph{is} virtually diffuse if $G$ is virtually diffuse, since it embeds in the virtually diffuse group $G * \Z$.
We now give a proof which uses CSP.
\begin{proof} [Second proof of \cref{notvirtdiff}]
As above we consider the GGS group defined by the vector $(1, \dots , 1, \lambda)$ and show that its torsion-free derived subgroup is not virtually diffuse. 
The orders of all congruence quotients of a GGS group acting on a $p-$rooted tree were computed exactly in \cite[Theorem A]{FZ14}. In particular, they are all powers of $p$.
\cite[Theorem A]{FGU17} shows that these GGS groups have the congruence subgroup property, which passes to finite index subgroups, so the derived subgroups have it as well and all finite quotients have order a power of $p$. 
By \cref{vb1} the derived subgroups aren't virtually locally indicable. 
We conclude residual finiteness and the non-virtual diffuseness as in the previous proof.
\end{proof} 
\begin{rmk}
    This covers in particular the Fabrykowski-Gupta group, defined by the vector $(1,0)$ for $p=3$. However, the generalised Fabrykowski-Gupta groups, defined by the vector $(1, 0, \dots, 0)$ and any integer $d \geq 3$, are not virtually torsion-free \cite[Proposition 4.1]{FGU17}: they have the congruence subgroup property (at least when $p \geq 7$ is a prime), are branched over their commutator subgroups, and the commutator subgroups aren't torsion-free.
\end{rmk}

\begin{rmk}
    Baer's conjecture is the conjecture that a group ring is Noetherian if and only if the group is virtually polycyclic. \cite{KL19} shows that a group with Noetherian group ring is amenable and Noetherian. \cite{Fi24} shows that there can be no counterexamples to Baer's conjecture among virtually locally indicable groups. By \cref{lidiff}, any counterexample must be a non-(virtually diffuse) group. \cref{notvirtdiff} exhibits infinitely many torsion-free examples of such groups. However, all GGS groups defined by non-constant sequences are branch groups \cite{FGU17}, which are not Noetherian and hence cannot be a counterexample to Baer's conjecture.
\end{rmk}
\begin{rmk}
    In \cite{FZ14} it is shown that GGS groups with non-constant defining vector are branched over the third derived subgroup $\gamma_3(G)$. This is not explicitly mentioned in the proof of \cref{notvirtdiff}, but it is used in \cite{FGU17} to show that the derived subgroup is torsion-free. 
\end{rmk}

\section{Open questions}
In \cite{KR16} the authors also ask the following variant:
\begin{question}
       Is there a torsion-free, finitely presented, residually finite group which isn't virtually diffuse?
\end{question}
In \cite{BEH08} an ascending $L-$presentation for the Fabrykowski-Gupta group is given, which gives rise to an embedding of it into a finitely presented group by taking HNN extensions. One could attempt to give an $L-$presentation for its finite index torsion-free subgroup $K$ and hence an embedding of $K$ into a finitely presented, torsion-free group which isn't virtually diffuse. However, these are likely not residually finite. In the case of the Grigorchuk group non-residual finiteness of the HNN extension was shown in \cite{SW02}.

It is also open whether there is a unique product group which isn't diffuse, which gives rise naturally to
\begin{question}
    Do the torsion-free commutator subgroups of the GGS groups in \cref{notvirtdiff} have unique products?
\end{question}
Computational evidence suggests they do.

\subsection*{Acknowledgements} The author thanks his advisor Giles Gardam for his guidance. This work was supported by the ERC Grant Satisfiability and group rings: SATURN, 101076148.

\end{document}